\documentclass[a4paper,10pt]{article}

\usepackage{flafter,epsfig,amsmath,amssymb,mathrsfs}

\usepackage{amsthm}

\usepackage{fullpage}

\def\R{\mathbf R} 

\renewcommand{\le}{\leqslant}
\renewcommand{\ge}{\geqslant}

\newtheorem{theorem}{Theorem}[section]         
\newtheorem{lemma}{Lemma}[section]             
\newtheorem{definition}{Definition}[section]   

\title{Three positive solutions of a nonlinear Dirichlet problem with competing power nonlinearities}
\author{Vladimir Lubyshev}

\begin{document}

\maketitle

\begin{abstract}
This paper studies a nonlinear Dirichlet problem for the $p$-Laplacian operator with nonlinearity consisting of 
power components. The problem under consideration can be though of as a perturbation of the Ambrosetti-Brezis-Cerami problem with concave-convex nonlinearity. The combined effect of power components in the perturbed nonlinearity allows to establish a higher order multiplicity of positive solutions. We study properties of the perturbed energy functional and prove the existence of three positive solutions to the problem at hand.
\end{abstract}

\section{Introduction}

In this paper, we study the following nonlinear Dirichlet problem:
\begin{equation}\label{e:main}
\left\{
\begin{aligned}
 -\Delta_p u &= \lambda \left( |u|^{\alpha-2}u - |u|^{\gamma-2}u \right) + |u|^{\beta-2}u && \mbox{ in } \Omega, \\
 u &= 0 && \mbox{ on } \partial\Omega.
\end{aligned}
\right.
\end{equation}
where $\Omega\subset\mathbf R^N$ is a bounded smooth domain, $\lambda$ is a real parameter, and $\Delta_p u = \nabla 
\cdot (|\nabla u|^{p-2}\nabla u)$ is the $p$-Laplacian operator. Problem \eqref{e:main} is studied under the 
following hypothesis on the exponents:

\begin{itemize}
 \item[\bf(H0)] $1 < \alpha < p < \beta < \gamma < p^*$,
\end{itemize}
where $p^*$ is the critical Sobolev exponent defined by
\[
p^* := \begin{cases} pN/(N-p) & \mbox{ if } p < N, \\ \infty & \mbox{ if } p \ge N. \end{cases}
\]
Our goal is to show that problem \eqref{e:main} has at least three positive solutions for any $\lambda\in(0,\lambda^*)$ 
where $\lambda^*$ is some positive real number. 

Multiplicity of positive solutions to nonlinear Dirichlet problem
\begin{equation}\label{e:gen}
\left\{
\begin{aligned}
 -\Delta_p u &= f_\lambda(\cdot,u) && \mbox{ in } \Omega, \\
 u &= 0 && \mbox{ on } \partial\Omega
\end{aligned}
\right.
\end{equation}
has been extensively studied by many authors. In \cite{brezis}, problem \eqref{e:gen} was studied with $p=2$, bounded 
$\Omega$, and concave-convex nonlinearity $f_\lambda(x,u) = \lambda |u|^{\alpha-2}u + |u|^{\beta-2}u$, 
$1<\alpha<p<\beta\le p^*$, where the existence of two positive solutions 
for $\lambda\in(0,\lambda^*)$, $\lambda^*>0$, has been proved. This result was generalized for $p\ge2$ and/or other nonlinear operators, and for more general 
nonlinearities essentially preserving the concave-convex structure (cf. 
\cite{ilyasov,zhou,kyritsi,ambrosetti}). It is known that, at least for the case where $\Omega$ is the unit 
ball, two positive solutions for the concave-convex problem considered in \cite{brezis} is the maximum one can expect 
(cf. \cite{tang}).

Some recent papers that study higher order multiplicity of positive solutions to nonlinear problems are \cite{hung,zhao,bai}. In \cite{hung}, a one-dimensional problem \eqref{e:gen} with $p=2$ and 
$f_\lambda(x,u) = \lambda g(u)$ has been considered for $\Omega=(-1,1)$. Under the assumption that $g$ is concave on 
$(0,\gamma)$ and convex on $(\gamma,\infty)$, $\gamma>0$, as well as other assumptions such as $g$ having a unique 
positive zero and $\lim_{u\to\infty} g(u)/u = \infty$, the existence of three positive solutions has been proved for 
$\lambda\in(\lambda_*,\lambda^*)$, $\lambda^*>\lambda_*>0$. In \cite{zhao}, problem \eqref{e:gen} was 
considered for $p>N$ and $f_\lambda(x,u) = \lambda a(x) u^{-\gamma} + \lambda g(x,u)$, $\gamma>0$. Under some 
assumptions on the coefficients and the exponent $\gamma$, the existence of three positive solutions has been proved 
for $\lambda$ belonging to an open subinterval of $(0,\infty)$. Finally, in \cite{bai}, a one-dimensional problem 
\eqref{e:gen} (the authors considered a more general form of the differential operator) with $f_\lambda(x,u) = \lambda 
a(x)g(u)$ and $\Omega=(0,1)$ has been considered. Under various assumptions, including $g(u)\le 
u^{p-1}$ in a positive neighborhood of 0, the existence of three positive 
solutions has been shown for $\lambda$ belonging to an open subinterval of $(0,\infty)$.

Problem \eqref{e:main}, which we study in this paper, is different from that in \cite{brezis} in the presence of the 
dominating negative term $-\lambda |u|^{\gamma-2}u$ which is responsible for producing an extra positive solution. This leads to the following main result of this paper.

\begin{theorem}\label{t:main} Under hypothesis (H0), there exists an $\lambda^*>0$ such that for any 
$\lambda\in(0,\lambda^*)$, 
problem \eqref{e:main} has three distinct positive solutions $u_1,u_2,u_3\in C^{1,\mu}(\bar\Omega)$, $0<\mu<1$, 
such that
\[
 \max(J_\lambda(u_1), J_\lambda(u_3)) < 0 < J_\lambda(u_2)
\]
where
\[
 J_\lambda(v) = \frac 1p \int_\Omega |\nabla v|^p - \frac \lambda\alpha \int_\Omega |u|^\alpha - \frac 1\beta 
\int_\Omega |u|^\beta + \frac \lambda\gamma \int_\Omega |u|^\gamma
\]
is an energy functional for problem \eqref{e:main}.
\end{theorem}

Many powerful methods exist in nonlinear analysis that help study the multiplicity of solutions to nonlinear differential equations. These methods include Morse theory, mountain pass lemma, fixed point theorems, and the Pohozaev fibering method to name a few (cf. \cite{chang,chang2,struwe,pohozaev_book} and references therein). The fibering method is especially useful when the nonlinearity contains polynomial components and will serve as the underlying method in this paper. For its applications to multiplicity of solutions to nonlinear equations and systems, cf. \cite{pohozaev2, bozhkov_mitidieri, ilyasov, lubyshev} and references therein.

This paper is organized as follows. Section \ref{sec:prelim} contains some known results to be used in the paper. 
Section \ref{sec:prop} studies properties of the energy functional, $J_\lambda$, necessary to prove Theorem 
\ref{t:main}. 
Sections \ref{sec:min1}--\ref{sec:min3} study three optimization problems, each leading to a positive solution of 
problem \eqref{e:main}. Finally, Section \ref{sec:proof} is dedicated to proving our main result, Theorem \ref{t:main}.

\section{Preliminaries}\label{sec:prelim}

In this section, we list two known results that help studying critical points of a functional defined on a Banach 
space. Throughout this section, differentiability always means Fr\'echet differentiability.

Let $(X,\|\cdot\|)$ be a Banach space, and let $B_\rho$ and $S_\rho$ be respectively its closed ball and sphere of 
radius $\rho$. We are concerned with critical points of a differentiable functional $J\colon X\to\R$. In other words, 
we want to study the equation
\begin{equation}\label{e:critpts}
 DJ(u) = 0, \qquad u\in X.
\end{equation}

\subsection{The fibering method}

In this subsection, we present the Fibering Method of S. Pohozaev \cite{pohozaev1,pohozaev2,pohozaev_book}. We will 
state its special form, the \emph{spherical fibering}, which is most suitable for our problem.

Suppose that the norm $\|\cdot\|$ of the Banach space $X$ is differentiable away from the origin and that
\[
 J = \frac 1p \|\cdot\|^p + R
\]
where $p>1$ and $R\colon X\to\mathbf R$ is a differentiable functional. 

Instead of studying critical points of the functional $J$ directly, constrained critical points of the fibering 
functional
\[
 I(v) = J(t(v)v)
\]
where $t=t(v)$ is a solution of
\[
 \frac \partial{\partial t} J(tv) = 0
\]
are studied on $S_1$. Under certain assumptions, a constrained critical point $v_*$ of $I$ corresponds to a critical point $u_*=t(v_*)v_*$ of the original functional $J$.

These constrained critical points typically arise as extrema of $I$ on $S_1$. But because $S_1$ is not weakly closed, 
the limit of a maximizing or minimizing sequence, if it exists, may no longer be in $S_1$. There is a regularity theorem that helps circumvent this difficulty. It reduces the problem to studying constrained critical points of an extension of $I$ to the unit ball. For reflexive $X$, this new problem is often more tractable because of the weak compactness of $B_1$.

A rigorous formulation of the (spherical) fibering method combined with the regularity theorem is presented below.

\begin{theorem}\label{t:spherical_fib} Let $\mathscr U$ be an open subset of $X$ with 
$\mathscr U\cap S_1\neq\varnothing$. Suppose that the equation
\begin{equation}\label{e:spherical_fib}
 t^{p-1} + DR(tv)v = 0
\end{equation}
has a solution $t\colon B_1\cap\mathscr U\to[0,\infty)$ that is differentiable on $(B_1\cap\mathscr U)\setminus\{0\}$. 
Consider the functional $I\colon (B_1\cap\mathscr U)\setminus\{0\}\to\mathbf R$ defined by
\[
 I(v) = \frac 1p t(v)^p + R(t(v)v)
\]
and let
\[
 \mathscr M := \{ v\in B_1\cap\mathscr U \mid t(v)\neq0 \}.
\]
Then for any critical point $v_*\in\mathscr M$ of $I$ on $B_1\cap\mathscr U$,
\begin{itemize}
 \item[\rm(a)] $v_*\in S_1$;
 \item[\rm(b)] $u_*=t(v_*)v_*$ is a critical point of $J$, provided that $DH(v_*)v_*\neq0$ where $H:=\|\cdot\|$.
\end{itemize}
\end{theorem}

\subsection{Mountain pass theorem}

Another result in the critical point theory that we will use is the Mountain Pass Theorem by A. Ambrosetti and P. 
Rabinowitz \cite{ambrosetti-rabinowitz}. For that theorem to be applied, a certain compactness condition needs to be 
satisfied, which is introduced is the definitions below.

\begin{definition}\rm Given $c\in\mathbf R$, a sequence $(u_n)\subset X$ is called a Palais-Smale sequence at level $c$ 
($\mathrm{(PS)}_c$-sequence, in short) for the differentiable functional $J\colon X\to\mathbf R$ if
\[
 J(u_n) \to c \qquad \mbox{ and } \qquad DJ(u_n) \to 0 \,\mbox{ in } X^*.
\]
\end{definition}

\begin{definition}\rm A differentiable functional $J\colon X\to\mathbf R$ is said to satisfy the Palais-Smale condition 
if given $c\in\mathbf R$, every $\mathrm{(PS)}_c$-sequence for $J$ contains a (strongly) convergent subsequence.
\end{definition}

The theorem is question is presented below.

\begin{theorem}\label{t:mountain_pass} Let $J\in C^1(X,\mathbf R)$ satisfy the Palais-Smale condition and let
\[
 \max(J(0),J(w)) < \inf_{S_\rho} J
\]
for some $\|w\|>\rho>0$. Then $J$ has a critical point $u_*$ such that
\[
 J(u_*) = \inf_{\gamma\in\Gamma} \max_{t\in[0,1]} J(\gamma(t))
\]
where
\[
 \Gamma := \{ \gamma\in C([0,1],X) \mid \gamma(0)=0\;\mbox{ and }\;\gamma(1)=w \}.
\]
\end{theorem}

\section{Some properties of the energy functional}\label{sec:prop}

It will be convenient for us to represent $J_\lambda$ as
\[
 J_\lambda(v) = \frac 1p \int_\Omega |\nabla v|^p - \frac \lambda\alpha A(v) - \frac 1\beta B(v) + \frac \lambda\gamma 
C(v)
\]
where
\[
 A(u) := \int_\Omega |u|^\alpha, \qquad B(u) := \int_\Omega |u|^\beta, \qquad C(u) := \int_\Omega |u|^\gamma.
\]

\begin{lemma}\label{l:coercive} The energy functional $J_\lambda$ is coercive for any $\lambda>0$.
\end{lemma}

\begin{proof}

It suffices to show that the functional
\begin{equation}\label{e:R_lambda}
 R_\lambda(v) := -\frac \lambda\alpha A(v) - \frac 1\beta B(v) + \frac \lambda\gamma C(v)
\end{equation}
is bounded from below on $W_0^{1,p}(\Omega)$. It is readily seen that
\[
 R_\lambda(v) \ge \frac \lambda\gamma \|v\|_{L^\gamma}^\gamma - c_1 \|v\|_{L^\gamma}^\beta - \lambda c_2 
\|v\|_{L^\gamma}^\alpha
\]
for some constants $c_1$ and $c_2$ independent of $v$. We complete the proof by noticing that the function
\[
 s \mapsto \frac \lambda\gamma s^\gamma - c_1 s^\beta - \lambda c_2 s^\alpha
\]
is bounded from below on $[0,\infty)$.

\end{proof}

Let
\[
 \hat J_\lambda(t,v) := \frac 1p t^p + R_\lambda(tv) \qquad \mbox{ and } \qquad \Phi_\lambda(t,v) := t^{-\alpha} \frac \partial{\partial t} \hat J_\lambda(t,v).
\]
Clearly, $\hat J_\lambda(t,v) = J_\lambda(tv)$ for $v\in S_1$. Consider the equation
\begin{equation}\label{e:bif_eq}
 \Phi_\lambda(t,v) = t^{p-\alpha} - t^{\beta-\alpha} B(v) + \lambda t^{\gamma-\alpha} C(v) - \lambda A(v) = 0,
 \qquad t > 0.
\end{equation}
Because of our assumption of the exponents, (H0), this equation has at most three solutions. Set
\[
 \mathscr U_\lambda := \{ v \in W_0^{1,p}(\Omega) \mid \mbox{ equation \eqref{e:bif_eq} has three solutions in $t$} \}.
\]
Given $v \in W_0^{1,p}(\Omega)$, the $i^{th}$ solution of \eqref{e:bif_eq} will be denoted $t_i(v) = t_{i,\lambda}(v)$, 
where the symbol $\lambda$ will be dropped for simplicity.

\begin{lemma}\label{l:J_prop} There exists an $\Lambda>0$ such that for any $\lambda\in(0,\Lambda)$, the following hold.
\begin{itemize}
\item[\rm(a)] $\mathscr U_\lambda\cap S_1\neq\varnothing$.
\item[\rm(b)] $\mathscr U_\lambda$ is open in $W_0^{1,p}(\Omega)$.
\item[\rm(c)] For any $v\in\mathscr U_\lambda\cap B_1$,
\[
 \hat J_\lambda(t_1(v),v) < 0 < \hat J_\lambda(t_2(v), v).
\]
There exists an $w\in \mathscr U_\lambda \cap S_1$ such that
\[
 J_\lambda(t_3(w)w) < 0.
\]
\end{itemize}
\end{lemma}

\begin{proof}

(a): By the Rellich-Kondrachov theorem, there is an $v^*\in S_1$ such that
\[
 B(v^*) = \max_{v\in S_1} B(v).
\]
Choose $\Lambda>0$ so that
\begin{equation}\label{e:Phi_pos}
 \max_{t>0} \left[ t^{p-\alpha} - t^{\beta-\alpha} B(v^*) \right] > \Lambda \max_{v\in S_1} A(v)
\end{equation}
and
\[
 \min_{t>0} \left[ t^{p-\alpha} - t^{\beta-\alpha} B(v^*) + \Lambda\, t^{\gamma-\alpha} C(v^*) \right] < 0.
\]
Then it is clear that $v^*\in\mathscr U_\lambda$ for all $\lambda\in(0,\Lambda)$.

(b): Consider the map
\[
 T_\lambda(t,v) := \left( \Phi_\lambda(t,v), \frac \partial{\partial t} \Phi_\lambda(t,v) \right).
\]
We will write $T_\lambda(t,v)>0$ (resp., $T_\lambda(t,v)<0$) to mean that the two components of $T_\lambda(t,v)$ are 
strictly positive (resp., strictly negative).

It is readily seen that $v\in\mathscr U_\lambda$ if and only if
\[
 T_\lambda(s_1,v) > 0 \qquad \mbox{ and } \qquad T_\lambda(s_2,v) < 0
\]
for some $0<s_1<s_2<\infty$. Now fix any $v_0\in\mathscr U_\lambda$ and choose $0<s_1^0<s_2^0<\infty$ with
\[
 T_\lambda(s_1^0,v_0) > 0 \qquad \mbox{ and } \qquad T_\lambda(s_2^0,v_0) < 0.
\]
Since $T_\lambda(t,\cdot)$ is continuous on $W_0^{1,p}(\Omega)$ for any $t>0$, there is a neighborhood $\mathscr 
N_{v_0}$ of $v_0$ in $W_0^{1,p}(\Omega)$ such that
\[
 T_\lambda(s_1^0,v) > 0 \qquad \mbox{ and } \qquad T_\lambda(s_2^0,v) < 0 \qquad
 \mbox{ for all } v \in \mathscr N_{v_0}.
\]
This implies that $\mathscr N_{v_0} \subset \mathscr U_\lambda$. Since $v_0$ was arbitrary, we conclude that 
$\mathscr U_\lambda$ is open.

(c): Reduce $\Lambda$, if necessary, so that
\begin{equation}\label{e:J_pos}
 \max_{t>0} \left[ \frac {t^{p-\alpha}}p - \frac {t^{\beta-\alpha}}\beta B(v^*) \right] > \frac \Lambda\alpha 
\max_{v\in S_1} A(v)
\end{equation}
and
\begin{equation}\label{e:J_neg}
 \min_{t>0} \left[ \frac {t^p}p - \frac {t^\beta}\beta B(v^*) + \Lambda \frac {t^\gamma}\gamma C(v^*) \right] < 0.
\end{equation}
Fix any $v\in \mathscr U_\lambda\cap B_1$. Then the inequality
\[
 \hat J_\lambda(t_1(v),v) < 0
\]
follows from the fact that $\alpha < \min(p,\beta,\gamma)$. It is clear that
\begin{equation}\label{e:J_incr}
 t \mapsto \hat J_\lambda(t,v) \qquad \mbox{ is increasing on } [t_1(v),t_2(v)].
\end{equation}

By \eqref{e:Phi_pos}, for any $v\in B_1\setminus\{0\}$, the equation
\[
 \tilde\Phi_\lambda(t,v) := t^{p-\alpha} - t^{\beta-\alpha} B(v) - \lambda A(v) = 0
\]
has exactly two solutions, $0 < \tilde t_1(v) < \tilde t_2(v) < \infty$.

Let us show that
\begin{equation}\label{e:t_vs_ttilde}
 t_1(v) < \tilde t_1(v) < \tilde t_2(v) < t_2(v).
\end{equation}
Since $\Phi_\lambda(t,v) > \tilde\Phi_\lambda(t,v)$ for all $t>0$,
\[
 \Phi_\lambda(\cdot,v)>0 \qquad \mbox{ on } [\tilde t_1(v),\tilde t_2(v)].
\]
Therefore, either $[\tilde t_1(v),\tilde t_2(v)] \subset (t_1(v),t_2(v))$ or $[\tilde t_1(v),\tilde t_2(v)] \subset 
(t_3(v),\infty)$. We want to verify that the second option is impossible. Indeed, since $\frac \partial{\partial t} 
\Phi_\lambda(t,v) > \frac \partial{\partial t} \tilde\Phi_\lambda(t,v)$ for $t>0$ and $\frac \partial{\partial t} 
\tilde\Phi_\lambda(t,v) > 0$ on $(0,\tilde t_1(v)]$,
\[
 \Phi_\lambda(\cdot,v) \qquad \mbox{ is increasing on } (0,\tilde t_1(v)].
\]
This implies that $\tilde t_1(v) < t_2(v)$, finishing the proof of \eqref{e:t_vs_ttilde}.

Entertaining \eqref{e:J_incr}, \eqref{e:t_vs_ttilde}, we deduce that
\[
 \hat J_\lambda(t_2(v),v) > \hat J_\lambda(\tilde t_2(v),v) \ge H(\tilde t_2(v),v)
\]
where
\[
 H(t,v) := \frac {t^p}p - \lambda \frac {t^\alpha}\alpha A(v) - \frac {t^\beta}\beta B(v).
\]
But by \eqref{e:J_pos},
\[
 H(\tilde t_2(v),v) = \max_{t>0} H(t,v) > 0,
\]
yielding that $\hat J_\lambda(t_2(v),v) > 0$.

Finally, take $w=v^*$. The minimum point for the left hand side of \eqref{e:J_neg} is bounded from below by a positive constant as 
$\Lambda\downarrow0$, whereas the first zero of $t\mapsto J_\lambda(tw)$ is $o(1)$ as $\lambda\downarrow0$. This means 
that we can decrease $\Lambda$, if necessary, so that the third critical value of $t\mapsto J_\lambda(tw)$ is negative 
or, equivalently, $J_\lambda(t_3(w)w) < 0$.

\end{proof}

\section{First critical point}\label{sec:min1}

Let $\Lambda>0$ be as in part (b) of Lemma \ref{l:J_prop}. In view of inequality \eqref{e:Phi_pos}, equation 
\eqref{e:bif_eq} has the minimal solution $t=t_1(v)$ for any $v\in B_1\setminus\{0\}$. Moreover,
\begin{equation}\label{e:min1}
 \left. \frac \partial{\partial t} \right|_{t=t_1(v)} \Phi_\lambda(t,v) > 0.
\end{equation}

\begin{lemma}\label{l:c1_dif} The functional $t_1$ is continuously differentiable and weakly continuous on 
$B_1\setminus\{0\}$.
\end{lemma}

\begin{proof}

The statement of the lemma follows from \eqref{e:min1}, the Implicit Function Theorem, and the weak continuity of 
$A,B$, and $C$.

\end{proof}

Define the functional $I_1\colon B_1\setminus\{0\}\to\mathbf R$ by
\[
 I_1(v) := \hat J_\lambda(t_1(v),v) < 0.
\]

\begin{theorem}\label{t:c1} For any $\lambda\in(0,\Lambda)$, there is a nonnegative $v_*\in S_1$ such that $u_* = t_1(v_*)v_*$ is a critical point of $J_\lambda$ with $J_\lambda(u_*)<0$.
\end{theorem}

\begin{proof}

Consider the problem
\begin{equation}\label{e:c1}
 c_1(\lambda) := \inf_{v\in B_1\setminus\{0\}} I_1(v).
\end{equation}
Let $(v_n)$ be its minimizing sequence, which we may assume to be nonnegative. According to Lemma \ref{l:coercive}, 
$t_{1,n} := t_1(v_n)$ is a bounded sequence. Therefore, without loss of generality we can assume that
\begin{align*}
 &v_n \rightharpoonup v_* \qquad \mbox{ weakly in } W^{1,p}(\Omega), \\
 &t_{1,n} \to t_*
\end{align*}
for some nonnegative $v_*\in B_1$ and $0\le t_*<\infty$.

Since $c_1(\lambda) < 0$, we must have $0<t_*<\infty$. In particular, equation \eqref{e:bif_eq} implies that 
$v_*\not\equiv0$. By Lemma \ref{l:c1_dif} and the weak closedness of $B_1$, the infimum in \eqref{e:c1} is attained at $v_*$.

Entertaining Theorem \ref{t:spherical_fib} and Lemma \ref{l:c1_dif}, we conclude that $v_*\in S_1$ and that $u_* = 
t_1(v_*)v_*$ is a critical point of $J_\lambda$. Since $v_*\in S_1$, $J_\lambda(u_*) = I_1(v_*) < 0$.

\end{proof}

\section{Second critical point}\label{sec:min2}

Consider the continuously differentiable functional $\bar J_\lambda\colon W_0^{1,p}(\Omega)\to\mathbf R$ defined by
\[
 \bar J_\lambda(v) := \frac 1p \int_\Omega |\nabla v|^p + R_\lambda(v^+)
\]
where $v^+:=\max(v,0)$.

\begin{lemma}\label{l:c2_palais} The following statements are true.
\begin{itemize}
 \item[\rm(a)] $\bar J_\lambda$ satisfies the Palais-Smale condition for any $\lambda>0$.
 \item[\rm(b)] There exists an $\rho>0$ such that for any $\lambda\in(0,\Lambda)$,
\[
 \max\left( \inf_{B_\rho} \bar J_\lambda, \inf_{W_0^{1,p}\setminus B_\rho} \bar J_\lambda \right) < 0 < \inf_{S_\rho} 
\bar J_\lambda.
\]
\end{itemize}
\end{lemma}

\begin{proof}

(a): Fix any $c\in\mathbf R$ and a $\mathrm{(PS)}_c$-sequence $(u_n)$ for $\bar J_\lambda$:
\begin{align}
 &\bar J_\lambda(u_n) \to c, \label{e:c2_1} \\ 
 &D\bar J_\lambda(u_n) \to 0 \qquad \mbox{ in } W^{-1,p'}(\Omega).\label{e:c2_2}
\end{align}
Mimicking the proof of Lemma \ref{l:coercive}, one verifies that the functional $\bar J_\lambda$ is coercive. Therefore 
\eqref{e:c2_1} implies that $(u_n)$ is bounded.

Since $\Delta_p\colon W_0^{1,p}(\Omega) \to W^{-1,p'}(\Omega)$ is a homeomorphism, we can write
\[
 (-\Delta_p)^{-1} D\bar J_\lambda(u_n) = u_n + (-\Delta_p)^{-1} DR_\lambda(u_n^+).
\]
By our assumptions (H0) on the exponents, $DR_\lambda\colon W_0^{1,p}(\Omega) \to W^{-1,p'}(\Omega)$ is a compact 
operator, and hence so is the operator $K_\lambda := (-\Delta_p)^{-1} DR_\lambda \colon W_0^{1,p}(\Omega) \to 
W_0^{1,p}(\Omega)$. Taking into account \eqref{e:c2_2}, we obtain
\[
 o(1) = u_n + K_\lambda(u_n^+).
\]
Since $(u_n)$ is bounded and $K_\lambda$ is compact, we conclude that $(u_n)$ contains a (strongly) convergent 
subsequence in $W_0^{1,p}(\Omega)$.

(b): Denote by $\rho>0$ the value of $t$ where the left hand side of \eqref{e:J_pos} attains its maximum. It is clear 
that
\[
 \inf_{B_\rho} \bar J_\lambda < 0 < \inf_{S_\rho} \bar J_\lambda.
\]
Since the value of $t$ where the left hand side of \eqref{e:J_neg} attains its minimum is $>\rho$, we 
also deduce that
\[
 \inf_{W_0^{1,p}(\Omega)\setminus B_\rho} \bar J_\lambda < 0.
\]

\end{proof}

\begin{theorem}\label{t:c2} For any $\lambda\in(0,\Lambda)$, the functional $J_\lambda$ has a critical point 
$u_*\ge0$ with $J_\lambda(u_*) > 0$.
\end{theorem}

\begin{proof}

According to statement (b) of Lemma \ref{l:c2_palais}, there is an $\rho>0$ and an $w\in 
W_0^{1,p}(\Omega)\setminus B_\rho$ such that
\[
 \max(\bar J_\lambda(0), \bar J_\lambda(w)) = 0 < \inf_{S_\rho} \bar J_\lambda.
\]
Since, by statement (a) of the same Lemma, $\bar J_\lambda$ satisfies the Palais-Smale condition, we conclude from the 
Mountain Pass Theorem \ref{t:mountain_pass} that $\bar J_\lambda$ has a critical point $u_*$ such that
\[
 \bar J_\lambda(u_*) = \inf_{\gamma\in\Gamma} \max_{t\in[0,1]} \bar J_\lambda(\gamma(t)) > 0
\]
where
\[
 \Gamma := \{ \gamma\in C([0,1],W_0^{1,p}(\Omega)) \mid \gamma(0)=0\;\mbox{and}\;\gamma(1)=w \}.
\]

To finish the proof, it suffices to show that $u_*\ge0$ because then $u_*$ will be a critical point of the original 
functional $J_\lambda$. Since
\[
 0 = D\bar J_\lambda(u_*) u_*^- = \int_\Omega |\nabla u_*|^{p-2} \nabla u_* \cdot \nabla u_*^- = \int_\Omega |\nabla 
u_*^-|^p
\]
and $u_*\in W_0^{1,p}(\Omega)$, we must have $u_*^-=0$ or, equivalently, that $u_*\ge0$.

\end{proof}

\section{Third critical point}\label{sec:min3}

Define the functional $I_3\colon\mathscr U_\lambda \to \mathbf R$ by
\[
 I_3(v) := \hat J_\lambda(t_3(v),v) = \min_{t>t_2(v)} \hat J_\lambda(t,v).
\]

\begin{lemma}\label{l:c3_dif} The functional $t_3$ is continuously differentiable and weakly continuous on $\mathscr 
U_\lambda$.
\end{lemma}

\begin{proof}

The statement of the lemma follows from the inequality
\[
 \left. \frac \partial{\partial t} \right|_{t=t_3(v)} \Phi_\lambda(t,v) > 0, \qquad v\in\mathscr U_\lambda
\]
the Implicit Function Theorem, and the weak continuity of $A,B$, and $C$.

\end{proof}

\begin{theorem}\label{t:c3} For any $\lambda\in(0,\Lambda)$, there is a nonnegative $v_*\in S_1\cap\mathscr U_\lambda$ 
such that $u_* = t_3(v_*)v_*$ is a critical point of $J_\lambda$ with $J_\lambda(u_*)<0$.
\end{theorem}

\begin{proof}

By statement (c) of Lemma \ref{l:J_prop},
\begin{equation}\label{e:c3}
 c_3(\lambda) := \inf_{v\in\mathscr U_\lambda\cap B_1} I_3(v) < 0.
\end{equation}

Let $(v_n)$ be a minimizing sequence for problem \eqref{e:c3} which we may assume to be nonnegative. Let $t_{3,n} 
:= t_3(v_n)$. Then we can assume that
\begin{align}
 &v_n \rightharpoonup v_* \qquad \mbox{ weakly in } W_0^{1,p}(\Omega), \\
 &t_{3,n} \to t_3
\end{align}
for some nonnegative $v_*\in B_1$ and $0\le t_3\le\infty$.

Since, by Lemma \ref{l:coercive}, $J_\lambda$ is coercive and $I_3(v_n) \ge J_\lambda(t_{3,n}v_n)$, we must have 
$t_3<\infty$. Entertaining \eqref{e:c3}, we also obtain that $0<t_3$. In other words, $0<t_3<\infty$. Taking into 
account equation \eqref{e:bif_eq}, we infer that $v_*\not\equiv0$.

Let us show that $v_*\in\mathscr U_\lambda$. Let $t_{2,n} := t_2(v_n)$. Without loss of generality, we may 
assume that
\[
 t_{2,n} \to t_2
\]
for some $0\le t_2\le t_3$. Since $v_*\not\equiv0$, we deduce from \eqref{e:bif_eq} that $t_2\neq0$ and hence
\[
 0 < t_2 \le t_3.
\]
It is also clear that both $t_2$ and $t_3$ are solutions to equation \eqref{e:bif_eq} for $v=v_*$:
\[
 \Phi_\lambda(t_i,v_*) = 0, \qquad i=2,3.
\]

By statement (c) of Lemma \ref{l:J_prop},
\[
 0 < \frac {t_{2,n}^p}p + R_\lambda(t_{2,n}v_n), \qquad n\ge1,
\]
yielding that
\begin{equation}\label{e:c3_t2}
 0 \le \frac {t_2^p}p + R_\lambda(t_2v_*).
\end{equation}
Since
\[
 \frac {t_3^p}p + R_\lambda(t_3v_*) = c_3(\lambda) < 0,
\]
we deduce that
\[
 0 < t_2 < t_3.
\]

Since $\alpha$ is the smallest exponent, we deduce from \eqref{e:c3_t2} that equation \eqref{e:bif_eq} for 
$v=v^*$ has a third solution, $t_1\in(0,t_2)$, satisfying
\[
 \frac {t_1^p}p + R_\lambda(t_1v_*) < 0.
\]
Therefore, equation \eqref{e:bif_eq} for $v=v_*$ has three solutions, yielding that $v_*\in\mathscr U_\lambda$. So, we 
deduce that the infimum in \eqref{e:c3} is attained at $v_*$.

Applying Theorem \ref{t:spherical_fib} and Lemma \ref{l:c3_dif}, we conclude that $v_*\in S_1$ and that 
$u_*=t_3(v_*)v_*$ is a critical point of $J_\lambda$. Since $v_*\in S_1$, $J_\lambda(u_*) = I_3(v_*) < 0$.

\end{proof}

\section{Proof of the main result}\label{sec:proof}

In this section, we prove the main result of this paper, Theorem \ref{t:main}. We will also need the following 
regularity result (cf. \cite{kandilakis}).

\begin{theorem}\label{t:regularity} Let $\Omega\subset\R^N$ be a bounded domain with $C^2$ boundary and let $u\in 
W_0^{1,p}(\Omega)$ be a weak solution of the problem
\[
\left\{
\begin{aligned}
 -\Delta_p u &= f(\cdot,u) && \mbox{ in } \Omega, \\
 u &= 0 && \mbox{ on } \partial\Omega.
\end{aligned}
\right.
\]
where $f\colon\Omega\times\mathbf R$ is a Carath\'eodory function that satisfies the inequality
\[
 |f(x,u)| \le c (1 + |u|^{m-1}) \qquad \mbox{ for a.e. } x\in\Omega \mbox{ and all } u\in\R,
\]
where $m=p^*$ if $p<N$ and $m>1$ if $p\ge N$. Then $u\in C^{1,\mu}(\bar\Omega)$ for some $\mu\in(0,1)$.
\end{theorem}

We now move on to proving the main result.

\begin{proof}[Proof of Theorem \ref{t:main}]

Denote by $\lambda^*$ the number $\Lambda$ in Lemma \ref{l:J_prop} and fix any $\lambda\in(0,\lambda^*)$.

Let $u_1,u_2,u_3$ be functions $u_*$ in Theorem \ref{t:c1}, Theorem \ref{t:c2}, and Theorem \ref{t:c3}, respectively. 
Since they are critical points of the energy functional, they are weak solutions of problem \eqref{e:main}. Since 
$J_\lambda(u_i) < 0 < J_\lambda(u_2)$ for $i=1,3$, $u_2\neq u_1$ and $u_2\neq u_3$. Let us show that $u_1\neq u_3$. If 
this was not the case, we would have $v_1=v_3$ and $t_1(v_1)=t_3(v_3)$. But then that would imply $t_1(v_3) = t_1(v_1) 
= t_3(v_3)$, contradicting the definition of $\mathscr U_\lambda$.

Applying Theorem \ref{t:regularity}, we infer that $u_1,u_2,u_3\in C^{1,\mu}(\bar\Omega)$ for some $\mu\in(0,1)$, which 
completes our proof.

\end{proof}

\end{document}